\documentclass[12 pt]{amsart}
\usepackage{amsmath}
\usepackage{amsthm}
\usepackage{amsfonts}
\usepackage{graphicx}
\usepackage{color}

\newtheorem{theorem}{Theorem}[section]
\newtheorem{corollary}[theorem]{Corollary}
\newtheorem{lemma}[theorem]{Lemma}

\theoremstyle{definition}
\newtheorem{definition}[theorem]{Definition}
\theoremstyle{remark}
\newtheorem{remark}[theorem]{Remark}
\numberwithin{equation}{section}

\newcommand{\II}{I\negmedspace I}

\begin{document}

\allowdisplaybreaks

\title[Mean curvature in manifolds with Ricci curvature bounded from below]
{Mean curvature in manifolds with Ricci curvature bounded from below}

\author[J. CHOE]{JAIGYOUNG CHOE}
\address{Jaigyoung Choe: Korea Institute for Advanced Study, Seoul, 02455, Korea}
\email{choe@kias.re.kr}

\author[A. Fraser]{Ailana Fraser}
\address{Ailana Fraser: Department of Mathematics, University of British Columbia, 121-1984 Mathematics Road,
Vancouver, BC V6T 1Z2, Canada}
\email{afraser@math.ubc.ca}

\thanks{J.C. supported in part by NRF 2011-0030044, SRC-GAIA, A.F. supported in part by NSERC}

\date{}

\begin{abstract}
 Let $M$ be a compact Riemannian manifold of nonnegative Ricci curvature and $\Sigma$ a compact embedded 2-sided minimal hypersurface in $M$. It is proved that there is a dichotomy: If $\Sigma$ does not separate $M$ then $\Sigma$ is totally geodesic and $M\setminus\Sigma$ is isometric to the Riemannian product $\Sigma\times(a,b)$, and if $\Sigma$ separates $M$ then the map $i_*:\pi_1(\Sigma)\rightarrow \pi_1(M)$ induced by inclusion is surjective. This surjectivity is also proved for  a compact 2-sided hypersurface with mean curvature  $H\geq(n-1)\sqrt{k}$ in a manifold of Ricci curvature $Ric_M\geq-(n-1)k,k>0$, and for a free boundary minimal hypersurface in a manifold of nonnegative Ricci curvature with nonempty strictly convex boundary. As an application it is shown that a compact $n$-dimensional manifold $N$ with the number of generators of $\pi_1(N)<n$ cannot be minimally embedded in the flat torus $T^{n+1}$.
\end{abstract}

\maketitle

\section{introduction}

Euclid's fifth postulate implies that there exist two nonintersecting lines on a plane. But the same is not true on a sphere, a non-Euclidean plane. Hadamard \cite{Ha} generalized this to prove that every geodesic  must meet every closed geodesic on a surface of positive curvature. Note that an $n$-dimensional minimal submanifold of a Riemannian manifold $M^k$ is a critical set for $n$-dimensional area. Replacing the geodesic with the minimal submanifold, Frankel \cite{F} further generalized Hadamard's theorem: Let $\Sigma_1$ and $\Sigma_2$ be immersed minimal hypersurfaces in a complete connected Riemannian manifold $M^{n+1}$ of positive Ricci curvature. If $\Sigma_1$ is compact, then $\Sigma_1$ and $\Sigma_2$ must intersect. It should be remarked that a manifold of nonnegative Ricci curvature like $\mathbb S^2\times\mathbb S^1$  has many disjoint minimal spheres.

Using the connectivity of the inverse image of $\Sigma_1$ under the projection map in the universal cover of $M$, Frankel also proved that the natural homomorphism of fundamental groups: $\pi_1(\Sigma_1)\rightarrow\pi_1(M)$ is surjective. This means that the minimality of $\Sigma_1$ imposes restrictions on $\pi_1(\Sigma_1)$. This reminds us of a similar restriction on $\pi_1(M)$, as proved by Myers \cite{M}, that if $M$ has positive Ricci curvature, then $\pi_1(M)$ is finite.

These two theorems of Frankel have the dual versions in the negatively curved case as follows:
If $M^n$ is a complete Riemannian manifold of nonpositive sectional curvature, then every compact immersed minimal submanifold $\Sigma^k$ must have an infinite fundamental group and moreover, if $\Sigma^k$ is totally geodesic, then $\pi_1(\Sigma^k)\rightarrow\pi_1(M^n)$ is 1-1 \cite{He}.

It was Lawson \cite{L} who first realized the topological implication of Frankel's theorem; he found that Frankel's proof of the surjectivity works also for each component of $M\setminus\Sigma$ if  $M$ is a compact connected orientable Riemannian manifold of positive Ricci curvature and $\Sigma$ is a compact embedded orientable minimal hypersurface. He then showed that $\pi_1(\bar{D}_j,\Sigma)=0,j=1,2$, where $M\setminus\Sigma=D_1\cup D_2$. This implies that $\Sigma$ has as many 1-dimensional holes (loops) as $D_j$ does. Hence $D_1$ and $D_2$ are handlebodies and since $M$ is diffeomorphic to $\mathbb S^3$,  $\Sigma$ is unknotted.

Recently Petersen and Wilhelm \cite{PW} gave a new proof of Frankel's generalized Hadamard theorem. They also showed that if $M$ has nonnegative Ricci curvature and has two nonintersecting minimal hypersurfaces, then these are totally geodesic and a rigidity phenomenon occurs. Whereas Frankel and Lawson used the second variation formula for arc length, Petersen and Wilhelm   utilized the superhamonicity of the distance function from a minimal hypersurface. It should be mentioned that Cheeger and Gromoll had used the superharmonicity of the distance function arising from a minimizing geodesic \cite{CG}. See also \cite{CK}, \cite{Wy}.

In this paper we show that there is a dichotomy for a compact Riemannian manifold of nonnegative Ricci curvature (Theorem \ref{heegaard}): A compact embedded 2-sided minimal hypersurface $\Sigma$ does not separate $M$ or separates $M$ into two nonempty components $D_1$ and $D_2$, and consequently, $\Sigma$ is totally geodesic and $M$ is isometric to a mapping torus or the map $i_*:\pi_1(\Sigma)\rightarrow \pi_1(\bar{D}_j),j=1,2,$ induced by inclusion is surjective. As a result $M$ cannot have more 1-dimensional holes than $\Sigma$ unless $M$ is diffeomorphic to $\Sigma\times\mathbb S^1$. The first part of Theorem \ref{heegaard} reminds us of the Cheeger-Gromoll splitting theorem \cite{CG}; their {\it line} is dual to our nonseparating minimal hypersurface. See also \cite{CaGa}.

The surjectivity of $i_*:\pi_1(\Sigma)\rightarrow \pi_1(\bar{D}_j)$ is obtained in more general settings as follows. Let $M^{n}$ be a Riemannian manifold  of Ricci curvature $Ric_M\geq -(n-1)k, \, k>0$ and let $\Sigma$ be a compact 2-sided hypersurface that bounds a connected region $\Omega$ in $M$. If $\Omega$ is mean convex with $H(\Sigma)\geq(n-1)\sqrt{k}$, then $\Sigma$ is connected and $i_*:\pi_1(\Sigma)\rightarrow \pi_1(\bar{\Omega}_i)$ is surjective. Thus if $n=3$ then $\Omega$ is a handlebody.

We also consider the case when the compact Riemannian manifold $M^n$ of nonnegative Ricci curvature has nonempty boundary $\partial M$ which is strictly convex with respect to the inward unit normal. Fraser and Li \cite{FL} showed that any two properly embedded orientable minimal hypersurfaces in $M$ meeting $\partial M$ orthogonally must intersect. They also showed that if $\Sigma$ is a properly embedded orientable minimal hypersurface in $M$ meeting $\partial M$ orthogonally, then $\Sigma$ divides $M$ into two connected components $D_1$ and $D_2$. Generalizing \cite{FL}, we show that the maps $i_*:\pi_1(\Sigma)\rightarrow \pi_1(\bar{M})$ and $i_*:\pi_1(\Sigma)\rightarrow \pi_1(\bar{D_j})$, $j=1,\,2$, are surjective and both components of $M\setminus\Sigma$ are handlebodies. When $n=3$ it is shown that $\Sigma$ is unknotted. We also prove some corresponding results in the case where $Ric_M\geq -(n-1)k, \, k>0$.

Finally, from our dichotomy (Theorem \ref{heegaard}) we derive nonexistence of some minimal embeddings. Let $N$ be an $n$-dimensional compact manifold with the number of generators of $\pi_1(N)=k$ that is minimally embedded in the flat $(n+1)$-torus $T^{n+1}$. Then we must have $k\geq n$. If $k=n$, then $N\approx T^n$, and if $k>n$, then $T^{n+1}\setminus N$ has two components $D_1,D_2$ such that the number of generators of $\pi_1(D_j)$ is bigger than $n,\,j=1,2$.  This is a higher dimensional generalization of Meeks' theorem \cite{Mk} that a compact surface of genus 2 cannot be minimally immersed in $T^3$.

\section{surjectivity}

It is well known that the second variation of arc length involves negative the integral of the sectional curvature. It is for this reason that the Ricci curvature affects both the mean curvature of the level surfaces of the distance function and the Laplacian of the distance function. The following lemma verifies this influence.

\begin{lemma} \label{lemma:levelsurface}
Assume that $M^{n+1}$ is a complete Riemannian manifold of nonnegative Ricci curvature. Let $D$ be a domain in $M$ and $N\subset\partial D$ a hypersurface with mean curvature $H_N \geq c$ with respect to the inward unit normal $\nu$ to $N$, i.e., $H_N=\langle \vec{H}_N,\nu\rangle$. Suppose that the distance function $d$ from $N$ is well defined in $D$. Then at a point $q \in D$ where $d$ is smooth\\
{\rm a)} the level surface of $d$ through $q$ has mean curvature $\geq c$ with respect to the unit normal away from $N${\rm (}in fact, that mean curvature is monotone nondecreasing in $d${\rm )}; \\
{\rm b)} $\Delta d\leq -c;$\\
{\rm c)} the level surfaces of $d$ are piecewise smooth, and at a nonsmooth point where two smooth level surfaces intersect, they make an angle $<\pi$ in the direction away from $N$.
\end{lemma}

\begin{proof}
Let $S$ be a smooth level surface of $d$ through a point $q\in D$. Let $\gamma\subset D$ be the geodesic up to $q$ that realizes the distance from $N$ and is parametrized by arc length. Then $\gamma$ hits $S$ and $N$ orthogonally at $q$ and at a point $p\in N$. Choose any unit vector $v$ tangent to $N$ at $p$ and parallel translate it along $\gamma$ to $q$, obtaining a unit parallel vector field $V$ along $\gamma$ which is normal to $\gamma$ and tangent to $S$ at $q$. Consider the lengths of the curves obtained by moving $\gamma$ in the direction of $V$. Then the second variation formula and the assumption that $S$ is a level surface of $d$ give us
$$L_V''={\II}_S(V,V)-{\II}_N(V,V)-\int_\gamma K(V,\gamma')\geq0,$$
where ${\II}$ denotes the second fundamental form defined by ${\II}(u,v)=\langle \nabla_u v,\nu \rangle$ with respect to the inward unit normal $\nu$ away from $N$, and $K(V,\gamma')$ is the sectional curvature on the span of $V$ and $\gamma'$. We can compute the same for orthonormal vectors $v_1,\ldots,v_n$ spanning the tangent space to $N$ at $p$ and sum up the above inequalities for the corresponding orthonormal parallel vector fields $V_1,\ldots,V_n$, to get
$$H_S(q)-H_N(p)-\int_\gamma {\rm Ric}(\gamma',\gamma')\geq0,$$
which proves a) because $\int_\gamma {\rm Ric}(\gamma',\gamma')$ is monotone nondecreasing in $d$. Let $E_1,\ldots,E_n$ be orthonormal vector fields on $S$ in a neighborhood of $q$. Extend them to orthonormal vector fields $\bar{E}_1,\ldots,\bar{E}_n,\bar{E}_{n+1}$ on $M$ in a neighborhood of $q$ such that $\bar{E}_{n+1}=\gamma'$. Then at $q$
\begin{equation*}\label{laplacian}
\Delta d=\sum_{i=1}^{n+1}\left[\bar{E}_i\bar{E}_i(d)-(\nabla_{\bar{E}_i}\bar{E}_i)d\right]=-H_S(q)\leq -H_N(p) \leq -c.
\end{equation*}
This proves b).

To prove c) suppose there are two distinct points $p_1$ and $p_2$ in $N$ with ${\rm dist}(p_i,q)=d(q)$. Let $N_i\subset N$ be a small neighborhood of $p_i$, $i=1,2$. Then $d_i(\cdot):={\rm dist}(\cdot,N_i)$ is a well defined function which is smooth in a neighborhood of $q$. Furthermore, it is easy to see that $d(\cdot)={\rm min}\{d_1(\cdot),d_2(\cdot)\}$. Therefore in a neighborhood of $q$ $$d^{-1}\{r\geq d(q)\}=d_1^{-1}\{r\geq d(q)\}\,\cap\, d_2^{-1}\{r\geq d(q)\}.$$
The proof is complete.
\end{proof}

It follows from Lemma \ref{lemma:levelsurface} that the distance function from a minimal hypersurface in a manifold of nonnegative Ricci curvature is superharmonic at points where it is smooth. In the following lemma we show that the distance function is {\em superharmonic in the barrier sense} at points where it is not smooth, and hence satisfies the maximum principle (\cite{C}, \cite{Pe} Theorem 66), that is, it is constant in a neighborhood of every local minimum.

\begin{lemma} \label{lemma:distance}
Let $\Sigma$ be a minimal hypersurface in a complete Riemannian manifold $M$ of nonnegative Ricci curvature. Then the distance function $d$ from $\Sigma$ is superharmonic $\Delta d \leq 0$ in the barrier sense. That is, given $p \in M$, for every $\varepsilon >0$ there exists a smooth support function from above $d_\varepsilon$ defined in a neighborhood of $p$ such that:
\begin{enumerate}
\item $d_\varepsilon(p)=d(p)$,
\item $d(x) \leq d_\varepsilon(x)$ in some neighborhood of $p$,
\item $\Delta d_\varepsilon (p) \leq \varepsilon$.
\end{enumerate}
\end{lemma}

\begin{proof}
By Lemma \ref{lemma:levelsurface} b) we know that $\Delta d \leq 0$ whenever $d$ is smooth. For any other $p \in M$ choose a unit speed minimizing geodesic $\gamma: [0,l] \rightarrow M$  between $\Sigma$ and $p$, with $\gamma(0) \in \Sigma$ and $\gamma(l)=p$.
Let $\nu$ be the unit normal of $\Sigma$ near $\gamma(0)$ in the direction toward $p$,
let $\varphi(t)=e^{-{t^2}/{(1-t^2)}}$ be a smooth cut-off function, and define
\[
      \Sigma_\delta = \{ \exp_x \delta \varphi( d_\Sigma(\gamma(0), x))  \nu(x) \; : \;
      x \in \Sigma \cap B_r (\gamma(0)) \}
\]
for small $\delta >0$, $r>0$. Since $\Sigma_\delta$ is a small perturbation of $\Sigma$, we have $|H_{\Sigma_\delta}|\leq C(\delta)$ with $C(\delta) \rightarrow 0$ as $\delta \rightarrow 0$. Given $\varepsilon >0$,  choose $\delta=\delta(\varepsilon)$ sufficiently small so that $C(\delta) \leq \varepsilon$. We claim that $d_\varepsilon(\cdot) := \delta(\varepsilon)+ d(\Sigma_{\delta(\varepsilon)}, \cdot)$ is a smooth support function from above for $d$ at $p$. It is clear from the construction that $d_\varepsilon(p)=d(p)$. If $x$ is sufficiently close to $p$, there is an interior point $x'$ in $\Sigma_\delta$ that realizes the distance from $x$ to $\Sigma_\delta$. By the construction of $\Sigma_\delta$, $d(\Sigma,x')\leq\delta$, and we have
\[
      d(x)=d(\Sigma,x) \leq d(\Sigma,x') + d(x',x)
             \leq \delta + d(\Sigma_\delta,x) = d_\varepsilon(x).
\]
If $d_\varepsilon$ is smooth at $p$, then by Lemma \ref{lemma:levelsurface} b), $\Delta d_\varepsilon(p) \leq C(\delta)\leq\varepsilon$.
It remains to show smoothness. Suppose $d_\varepsilon$ is not smooth at $p$. Then we know that either
\begin{enumerate}
\item there are two minimizing geodesics from $p$ to $\Sigma_\delta$, or
\item $p$ is a focal point of $\Sigma_\delta$.
\end{enumerate}
In case (1), there is a minimizing geodesic from $p$ to a point $q \neq \gamma(\delta)$ in $\Sigma_\delta$. But by construction of $\Sigma_\delta$, $d(\Sigma, q)<\delta$, and so this implies that
\[
     d(\Sigma,p) \leq d(\Sigma,q)+d(q,p) < \delta +d(\Sigma_\delta,p)=l,
\]
a contradiction.
In case (2), if $p$ is a focal point of $\Sigma_\delta$, there is a Jacobi field $J$ along $\gamma|_{[\delta, l]}$ with $J(\delta)$ tangent to $\Sigma_\delta$ at $\gamma(\delta)$, $J(l)=0$, and  such that $J'(\delta) + S_{\gamma'(\delta)} (J(\delta))$ is orthogonal to $\Sigma_\delta$, where $S_{\gamma'(\delta)}$ is the linear operator on $T_{\gamma(\delta)} \Sigma_\delta$ given by the second fundamental form of $\Sigma_\delta$ in $M$, that is, $S_{\gamma'(\delta)}X=-\left(\nabla_X\gamma'(\delta)\right)^T,X\in T_{\gamma(\delta)}\Sigma_\delta$.
The second variation of length of $\gamma|_{[\delta, l]}$ in the direction $J$ is zero:
\begin{align*}
   I(J,J)&=\int_\delta^l \left[|\nabla_{\gamma'} J|^2  - \langle R(J,\gamma')\gamma',J \rangle\right] \, dt
                    + \langle \nabla_J J, \gamma' \rangle \Big|_\delta^l \\
             &=-\int_\delta^l \langle J''+ R(J,\gamma')\gamma',J \rangle \, dt
                    + [ \langle \nabla_{\gamma'} J, J \rangle + \langle \nabla_J J, \gamma' \rangle ] \Big|_\delta^l \\
             &=-\int_\delta^l \langle %\nabla_{\gamma'} \nabla_{\gamma'} J
                      J''+ R(J,\gamma')\gamma',J \rangle \, dt
                   -\langle J'(\delta) + S_{\gamma'(\delta)} (J(\delta)) , J(\delta) \rangle \\
              & =0.
\end{align*}
Let $\sigma$ be the geodesic in $\Sigma_\delta$ with $\sigma(0)=\gamma(\delta)$ and $\sigma'(0)=J(\delta)$. For $\delta$ small, there is a unique minimizing geodesic $\gamma_s$ between $\sigma(s)$ and $\Sigma$, and since  $\gamma(\delta)$ is the point on $\Sigma_\delta$ that is furthest from $\Sigma$,
\[
          \left.\frac{d^2}{ds^2}\right|_{s=0} L(\gamma_s) <0.
\]
 Let $W$ be the variation field of the variation $\gamma_s$ of $\gamma|_{[0,\delta]}$. Then $W(\delta)=J(\delta)$, and for the vector field $V$ along $\gamma$ given by
\begin{align*}
      V(t)=\begin{cases} W(t)  & \mbox{ for } 0 \leq t \leq \delta \\
                                                 J(t)    & \mbox{ for } \delta \leq t \leq l, \end{cases}
\end{align*}
the second variation of length of $\gamma$ is strictly less than zero. This contradicts the fact that $\gamma$ is a minimizing geodesic from $p$ to $\Sigma$.
Therefore, $d_\varepsilon$ is smooth in a neighborhood of $p$ and is a smooth support function from above for $d$ at $p$.
\end{proof}

With the superharmonicity of the  distance function in our hands we are now able to prove the main theorem.

\begin{definition}
Let $\Sigma$ be a compact connected embedded hypersurface in a compact $n$-manifold $M$. $\Sigma$ is said to be {\it separating} if $M\setminus\Sigma$ has two nonempty connected components, and {\it nonseparating} if $M\setminus\Sigma$ is connected.
\end{definition}

\begin{definition}
A {\it handlebody} is a 3-manifold with boundary which is homeomorphic to a closed regular neighborhood of a connected properly embedded 1-dimensional CW complex in $\mathbb R^3$. A surface $\Sigma$ in a 3-manifold $M$ is called a {\it Heegaard surface} if $\Sigma$ separates $M$ into two handlebodies.
\end{definition}

\begin{theorem}\label{heegaard}
Let $M$ be a compact Riemannian n-manifold of nonnegative Ricci curvature and $\Sigma$ a compact connected embedded 2-sided minimal hypersurface in $M$. Then either
\begin{enumerate}
\item[a)]
$\Sigma$ is nonseparating and totally geodesic and $M$ is %a Riemannian product $\Sigma\times\mathbb S^1$.\\
isometric to a mapping torus
\[
          \frac{\Sigma \times [0,a]}{(x,0) \sim (y,a) \mbox{ iff } \phi(x)=y},
\]
where $\phi: \Sigma \rightarrow \Sigma$ is an isometry, or
\item[b)]
$\Sigma$ is separating, and if $D_1,D_2\subset{M}$ are the components of ${M}\setminus{\Sigma}$, then for $j=1,2$ the maps
\[
       i_*: \pi_1(\Sigma) \rightarrow \pi_1(\bar{D}_j),\,\,\,\,i_*: \pi_1(\Sigma) \rightarrow \pi_1(M)\,\,\,\,{ and}\,\,\,\,i_*: \pi_1(D_j)\rightarrow \pi_1(M)
\]
induced by the inclusion are all surjective.
\end{enumerate}
If $n=3$ and $\Sigma$ is a separating, then $\Sigma$ is a Heegaard surface.
\end{theorem}

\begin{proof}
a) Choose a function that is equal to 0 on $\Sigma$ and in a neighborhood of one side of $\Sigma$ and equal to 1 in a neighborhood of the other side of $\Sigma$. Since $\Sigma$ is nonseparating, this function can be extended to a smooth function on $M \setminus \Sigma$, and by passing to the quotient mod $\mathbb{Z}$ we obtain a nonconstant smooth function
\[
       f : M \rightarrow \mathbb{R} / \mathbb{Z} = \mathbb{S}^1.
\]
Let $f_*: \pi_1(M) \rightarrow \mathbb{Z}$ be the induced map on the fundamental groups, and for the universal cover $\tilde{M}$ of $M$, consider the cyclic cover $\hat{M}=\tilde{M}/G$ of $M$ corresponding to the subgroup $G=\ker f_*$ of $\pi_1(M)$. Since $\hat{M}$ has a geodesic line, the result follows from the splitting theorem \cite{CG}. However, we will give an alternate direct proof, which will be needed for the proof of part b).

Let $\Sigma_1,\Sigma_2\subset\hat{M}$ be two adjacent preimages of $\Sigma$ under the projection $\pi: \hat{M} \rightarrow M$ such that $\Sigma_1$ and $\Sigma_2$ bound a connected domain $D\subset\hat{M}$ on which $\pi$ is 1-1. Here we adopt the arguments of \cite{PW}. If $d_i$ is the distance function on $D$ to $\Sigma_i$, then our hypotheses on the Ricci curvature of $M$ and the minimality of $\Sigma_i$ imply that $\Delta d_i\leq0$ in the barrier sense, by Lemma \ref{lemma:distance}. Hence $d_1+d_2$ is also superharmonic in the barrier sense. But it has an interior minimum on a minimal geodesic $\gamma$ between $\Sigma_1$ and $\Sigma_2$ and so by the maximum principle it is constant on $D$. Then it follows that $d_i$ is harmonic and smooth on $D$. Recall the Bochner formula for a smooth function $u$ on $\hat{M}$:
\begin{equation*}\label{bochner}
\frac{1}{2}\Delta|du|^2=|{\rm Hess}\,u|^2+\langle\nabla u,\nabla(\Delta u)\rangle+{\rm Ric}(\nabla u,\nabla u).
\end{equation*}
Since $|du|=1$ for $u=d_i$, the formula yields ${\rm Hess}\,d_i=0$ on $D$. Therefore $\Sigma_i$ is totally geodesic and $\hat{M}$ is isometric to $\Sigma_i\times\mathbb R^1$. Thus $M$ is isometric to a mapping torus
$$ \frac{\Sigma \times [0,a]}{(x,0) \sim (y,a) \mbox{ iff } \phi(x)=y}\,,$$
where $a$ is the length of $\gamma$ and $\phi:\Sigma\rightarrow\Sigma$ is an isometry.

b) Let $\pi:\tilde{D}_j\rightarrow D_j$ be the universal cover of $D_j,\,j=1,2$.  Extending $\pi$ up to $\partial\tilde{D}_j$, we claim that $\partial \tilde{D}_j=\pi^{-1}(\Sigma)$ is connected. If $\partial \tilde{D}_j$ is not connected, let
\[
         d_0=\inf\{ d(\Sigma',\Sigma''): \Sigma' \mbox{ and } \Sigma''
         \mbox{ are distinct components of } \partial \tilde{D}_j \}.
\]
As in \cite{L}, there exist components $\Sigma'$ and $\Sigma''$ such that there is a geodesic $\gamma$ in $\tilde{D}_j$ from $\Sigma'$ to $\Sigma''$ of length $d_0$. By continuity, there is a neighborhood of $\gamma$ in $\tilde{D}_j$ such that the distance functions $d'$ and $d''$ to $\Sigma'$ and to $\Sigma''$ in $\tilde{D_j}$ are well defined. By the same arguments  as in a) we see that $d'+d''$ is superharmonic. Note that $d'+d''$ has interior minimum at all points of $\gamma$. As in a), it follows that a neighborhood of $\gamma$ is isometric to a product manifold $(\Sigma' \cap U) \times(0,d_0)$, where $U$ is a neighborhood of $\gamma(0)$ in $\Sigma'$.
Let $\mathcal{U}$ be the set of points in $\Sigma'$ that can be connected to $\Sigma''$ by a geodesic of length $d_0$. By the argument above, $\mathcal{U}$ is open and $U \subset \mathcal{U}$. We claim that $\mathcal{U}$ is also closed. To see this, let $p_m$ be a sequence of points in $\mathcal{U}$ converging to $p \in \Sigma'$, and let $\gamma_m$ be a geodesic in $\tilde{D}_j$ of length $d_0$ from $p_m$ to $\Sigma''$. By passing to a subsequence we can see that there exists a geodesic $\gamma_0$ of length $d_0$ from $p$ to $\Sigma''$ such that $\{\gamma_m\}$ converges to $\gamma_0$. It may happen that $\gamma_0$ hits $\partial\tilde{D}_j\setminus\Sigma'$ at some point $q$ with ${\rm dist}(p,q)<d_0$. But then dist$(p,\partial\tilde{D}_j\setminus\Sigma')<d_0$, which is a contradiction. Therefore $p \in \mathcal{U}$ and $\mathcal{U}$ is closed. Since $\mathcal{U}$ is both open and closed, $\mathcal{U}=\Sigma'$. Therefore $\tilde{D}_j$ is isometric to the product manifold $\Sigma'\times(0,d_0)$. Hence  $M\setminus\Sigma$ is isometric to $\Sigma\times(0,d_0)$, and so $M$ is diffeomorphic (not necessarily isometric) to $\Sigma\times\mathbb S^1$. But then $\Sigma$ is nonseparating in $M$, which is a contradiction. Hence $\partial \tilde{D}_j$ is connected, as claimed.

Let $\ell$ be a loop in ${\bar{D}_j}$ with base point $p\in\Sigma$. Lift $\ell$ to a curve $\tilde{\ell}$ in $\tilde{D}_j$ from $p_1\in\pi^{-1}(p)$ to $p_2\in\pi^{-1}(p)$. Since $\pi^{-1}(\Sigma)$ is connected, there is a curve $\hat{\ell}$ in $\pi^{-1}(\Sigma)$ connecting $p_1$ to $p_2$. Moreover, $\hat{\ell}$ is homotopic to $\tilde{\ell}$ in $\tilde{D}_j$ as $\tilde{D}_j$ is simply connected. Hence $\pi(\hat{\ell})$ is a loop in $\Sigma$ that is homotopic in $D_j$ to $\ell$. Therefore the map $i_*:\pi_1(\Sigma)\rightarrow \pi_1(\bar{D}_j)$ induced by inclusion is surjective.

Let $\ell$ be a loop in $M$. Divide $\ell$ into two parts $\ell_1,\ell_2$ such that $\ell_j\subset D_j$. Cover $\ell_j$ with a curve $\tilde{\ell}_j$ in $\tilde{D}_j$ as above. By the connectedness of $\pi^{-1}(\Sigma)$ again we have a curve $\hat{\ell}_j$ in $\pi^{-1}(\Sigma)$ with the same end points as $\tilde{\ell}_j$ and homotopic to $\tilde{\ell}_j$. Then $\pi(\hat{\ell}_1)\cup\pi(\hat{\ell}_2)$ is a loop in $\Sigma$ that is homotopic to $\ell$. Hence $i_*:\pi_1(\Sigma)\rightarrow \pi_1(M)$ is also surjective.

Since $\Sigma$ is 2-sided there exists $\bar{\ell}_j\subset D_j$ which is very close to $\pi(\hat{\ell}_1)\cup\pi(\hat{\ell}_2)$. Hence $\bar{\ell}_j$ and $\pi(\hat{\ell}_1)\cup\pi(\hat{\ell}_2)$ are homotopic and therefore $i_*:\pi_1(D_j)\rightarrow \pi_1(M)$ is surjective.

To prove the final statement, suppose $n=3$. The surjectivity of $i_*:\pi_1(\Sigma)\rightarrow \pi_1(\bar{D}_j)$ implies $\pi_1(\bar{D}_j,\Sigma)\approx0$. Then by \cite{Pa} we can use the Loop Theorem and Dehn's Lemma to show that $\bar{D}_1$ and $\bar{D}_2$ are handlebodies. Thus $\Sigma$ is a Heegaard surface.
\end{proof}

\begin{corollary}\label{unknotted}
Let $\Sigma$ be a compact connected embedded minimal surface in a Riemannian three-sphere $M$ of nonnegative Ricci curvature. Then $\Sigma$ is unknotted.
\end{corollary}

\begin{proof}
From the Jordan-Brouwer separation theorem it follows that $\Sigma$ is separating. We show that $\Sigma$ is unknotted in the sense that if $\Sigma'$ is a standardly embedded surface of the same genus as $\Sigma$ in $M$, then there exists an orientation preserving diffeomorphism $f:M \rightarrow M$ such that $f(\Sigma)=\Sigma'$.
By Theorem \ref{heegaard} b),  $\Sigma$ is a Heegaard surface. It follows from \cite{W} that there is a PL homeomorphism $\tilde{f}: M \rightarrow M$ such that $\tilde{f}(\Sigma)=\Sigma'$. Then by results from \cite{HM}  there exists a smooth map $f$ as claimed.
\end{proof}

\begin{remark} It should be mentioned that Meeks and Rosenberg \cite{MR} showed a noncompact properly embedded minimal surface in $\mathbb S^2\times\mathbb R$ is unknotted.
\end{remark}

The result of Frankel shows that two compact minimal hypersurfaces in a manifold of positive Ricci curvature must intersect. However, a manifold of nonnegative Ricci curvature can have many disjoint compact minimal hypersurfaces. Furthermore, in the case of negative curvature, there can even exist disjoint hypersurfaces that bound a {\em mean convex} region; for example, spheres equidistant to two disjoint planes in hyperbolic space. On the other hand two disjoint horospheres in hyperbolic space cannot bound a mean convex region. This suggests that there can only exist a mean convex region with two disjoint boundary components if the mean curvature is less than a critical number involving a lower bound on the curvature of the ambient manifold.

Here we show this, that in fact Frankel's argument can be extended to the case of manifolds of negative Ricci curvature provided the Ricci curvature is bounded from below and the hypersurfaces have mean curvature that is sufficiently large. We obtain the corresponding result on surjectivity of the natural homomorphism of fundamental groups for compact 2-sided hypersurfaces with mean curvature above this critical (sharp) threshold involving the lower bound on the Ricci curvature. In the 3-dimensional case, such hypersurfaces must bound handlebodies; for example, a compact connected 2-sided hypersurface with mean curvature $|H| \geq 2$ in hyperbolic 3-space bounds a handlebody.

\begin{theorem} \label{theorem:negative}
Let $M^n$ be a complete Riemannian manifold of Ricci curvature bounded from below, $Ric_M \geq -(n-1)k$, $k >0$.
Let $\Sigma$ be a compact hypersurface that bounds a connected region $\Omega$ in $M$. Suppose that the mean curvature vector of $\Sigma$ points everywhere into $\Omega$, and $H \geq (n-1)\sqrt{k}$. Then $\Sigma$ is connected, and the map
\[
       i_*: \pi_1(\Sigma) \rightarrow \pi_1(\bar{\Omega})
\]
induced by the inclusion is surjective. If $n=3$ then $\Omega$ is a handlebody.
\end{theorem}

\begin{proof}
We argue by contradiction. Suppose $\Sigma$ is not connected. Let $\Sigma_1$ and $\Sigma_2$ be distinct connected components of $\Sigma$. Then there exists a unit speed geodesic $\gamma: [-l/2,l/2] \rightarrow M$ with $\gamma(-l/2)=p \in \Sigma_1$ and $\gamma(l/2)=q \in \Sigma_2$ that realizes distance from $\Sigma_1$ to $\Sigma_2$, and meets $\Sigma$ orthogonally at the endpoints on the mean convex side of $\Sigma$. Let $e_1, \ldots, e_{n-1}$ be an orthonormal basis for the tangent space to $\Sigma_1$ at $p$, and parallel transport to obtain parallel orthonormal vector fields $E_1, \ldots, E_{n-1}$ along $\gamma$. Since $\gamma$ meets $\Sigma_2$ orthogonally, $E_1(q), \ldots, E_{n-1}(q)$ are tangent to $\Sigma_2$ at $q$. Let $V_i(t)=\varphi(t)E_i(t)$ with $\varphi(t)=\frac{1}{c(l)} \cosh(\sqrt{k}\,t)$ and $c(l)=\cosh(\sqrt{k}\,l/2)$. Note that $\varphi''-k\varphi=0$ and $\varphi(-l/2)=\varphi(l/2)=1$. Consider the sum of the second variations of length of $\gamma$ in the directions $V_i$:
\begin{align*}
     0 &\leq \sum_{i=1}^{n-1}L_{V_i}''(0) \\
     &=\int_{-l/2}^{l/2} \left[ (n-1)(\varphi')^2-\varphi^2 \mbox{Ric}(\gamma'\gamma') \right]dt
                      +\sum_{i=1}^{n-1}\varphi^2 \langle \nabla_{E_i}E_i, \gamma' \rangle\Big|_{-l/2}^{l/2} \\
     &= -\int_{-l/2}^{l/2} \left[  (n-1) \varphi \varphi'' + \varphi^2 \mbox{Ric}(\gamma'\gamma') \right] dt
                        -H_\Sigma(p)-H_\Sigma(q) + (n-1)\varphi \varphi'\Big|_{-l/2}^{l/2} \\
     & \leq -(n-1) \int_{l/2}^{l/2} \varphi (\varphi'' - k\varphi) \, dt
          -2(n-1)\sqrt{k} + 2(n-1)\sqrt{k}\, \tanh(\sqrt{k} \,l/2) \\
     & = -2(n-1)\sqrt{k} + 2(n-1)\sqrt{k}\, \tanh(\sqrt{k} \,l/2) \\
     & < 0,
\end{align*}
which is a contradiction. Therefore $\Sigma$ is connected. Similarly $\pi^{-1}(\Sigma)$ is connected in the universal cover $\tilde{\Omega}$ of $\Omega$ under the covering map $\pi:\tilde{\Omega}\rightarrow\Omega$. It then follows by arguments as in the proofs of Theorem \ref{heegaard} b) that $i_*: \pi_1(\Sigma) \rightarrow \pi_1(\bar{\Omega})$ is surjective, and if $n=3$ then $\Omega$ is a handlebody.
\end{proof}

\begin{remark}
The assumption that $\Sigma$ bounds a region is not necessary. If $M^n$ is a complete Riemannian manifold with $Ric_M \geq -(n-1)k$, $k >0$, and $\Sigma$ is a compact 2-sided hypersurface in $M$ with $|H| \geq (n-1)\sqrt{k}$, then it follows that $\Sigma$ bounds a collection of disjoint connected regions $\Omega_1, \ldots, \Omega_s$ in $M$ such that the mean curvature vector $\vec{H}$ points everywhere into $\Omega_i$, and each has as boundary $\partial \Omega_i$  a connected component of $\Sigma$. To see this, first observe that each component $\Sigma'$ of $\Sigma$ is separating. If not, we may construct a cyclic cover $\hat{M}$ of $M$ as in the proof of Theorem \ref{heegaard} a). Then $Ric_{\hat{M}} \geq -(n-1) k$, and each component of $\pi^{-1}(\Sigma')$ divides $\hat{M}$ into two infinite pieces. For one of these pieces, $\Omega$, the mean curvature vector $\vec{H}$ of $\partial \Omega$ points everywhere into $\Omega$ and satisfies $|H| \geq (n-1)\sqrt{k}$. It follows from \cite{S} Lemma 1 that
\[
         \mbox{Vol}(\Omega) \leq \frac{1}{n-1} \mbox{Vol}(\partial \Omega) < \infty,
\]
a contradiction. Therefore each component of $\Sigma$ is separating, and hence $\Sigma$ bounds a collection of disjoint regions $\Omega_1, \ldots, \Omega_s$ such that  $\vec{H}$ points everywhere into $\Omega_i$ for $i=1, \ldots, s$. Finally, Theorem \ref{theorem:negative} implies that $\partial \Omega_i$ is connected for each $i$ and hence is a connected component of $\Sigma$.
\end{remark}

\begin{remark}
This theorem is sharp in the sense that on a hyperbolic surface, disjoint circles of curvature 1 cannot bound a convex region, but on a hyperbolic surface with a cusp there exists a convex annular region with two boundary components (cross sections of the cusp) that have curvature slightly less than, but arbitrarily close to 1. One can construct analogous {\em compact} examples in higher dimensions in quotients of hyperbolic space $\mathbb{H}^n$. Disjoint horospheres with $H=n-1$ cannot bound a mean convex region in hyperbolic $n$-space $\mathbb{H}^n$, but in the half-space model of $\mathbb H^n$ there can exist a convex region bounded by the two hyperplanes $P_1,P_2$ with $\partial P_1=\partial P_2\subset\partial\mathbb H^n$ and making angles $\theta,\pi-\theta$ with $\partial\mathbb H^n$; the boundary components have mean curvature slightly less than, but arbitrarily close to $n-1$ as $\theta\rightarrow0$.
\end{remark}

\section{convex domain}

In this section the Riemannian manifold $M^n$ will be assumed to have nonempty boundary $\partial M$. Suppose that $M$ has nonnegative Ricci curvature and $\partial M$ is strictly convex. Recall that Frankel \cite{F} showed that two compact immersed minimal hypersurfaces in a Riemannian manifold $M$ of positive Ricci curvature must intersect. Fraser and Li (\cite{FL}, Lemma 2.4) extended Frankel's theorem to two properly embedded minimal hypersurfaces $\Sigma_1,\Sigma_2$ in $M$, i.e., $\partial\Sigma_i\subset\partial M,i=1,2$,  meeting $\partial M$ orthogonally. They also showed (\cite{FL} Corollary 2.10) that if $\Sigma$ is a properly embedded orientable minimal hypersurface in $M$ meeting $\partial M$ orthogonally, then $\Sigma$ divides $M$ into two connected components $D_1$ and $D_2$. We show that the maps $i_*:\pi_1(\Sigma)\rightarrow \pi_1(\bar{D_j})$, $j=1,\, 2$, are surjective and that $\Sigma$ is unknotted when $n=3$. We also prove some corresponding results in the case where the Ricci curvature is bounded from below by a negative constant.

\begin{lemma}\label{surjective}
Let $M$ be an $n$-dimensional compact manifold of nonnegative Ricci curvature with strictly convex boundary $\partial M$. Suppose that $\Sigma$ is a properly embedded minimal hypersurface in $M$ meeting $\partial M$ orthogonally. Then  the maps $i_*:\pi_1(\Sigma)\rightarrow \pi_1(M)$ and $i_*:\pi_1(\Sigma)\rightarrow \pi_1(\bar{D_j})$, $j=1,\,2$, are surjective, where $D_1,D_2$ are the components of $M \setminus \Sigma$.
\end{lemma}

\begin{proof}
Let $\tilde{D_j}$ be the universal cover of $D_j$ with the projection map  $\pi:\tilde{D_j} \rightarrow D_j$. Since  $\pi^{-1}(\Sigma)$ is connected, by the same arguments as in the proof of Theorem \ref{heegaard} b) we easily get the surjectivity of $i_*:\pi_1(\Sigma)\rightarrow \pi_1(\bar{D_j})$. Applying the same arguments to $\pi:\tilde{M}\rightarrow M$, we get the surjectivity of $i_*:\pi_1(\Sigma)\rightarrow \pi_1(M)$ as well.
\end{proof}

\begin{theorem} \label{convex-handlebodies}
Let $M$ be a $3$-dimensional compact orientable Riemannian manifold of nonnegative Ricci curvature. Suppose $M$ has nonempty boundary $\partial M$ which is strictly convex with respect to the inward unit normal. Then an orientable properly embedded minimal surface $\Sigma$ in $M$ meeting $\partial M$ orthogonally divides $M$ into two handlebodies.
\end{theorem}

\begin{proof}
By Lemma \ref{surjective} we have $\pi_1(\bar{D}_1,\Sigma)=\pi_1(\bar{D}_2,\Sigma)=0$. As in the proof of Theorem \ref{heegaard}, using the Loop Theorem and Dehn's Lemma, we conclude that $D_1$ and $D_2$ are handlebodies.
\end{proof}

\begin{corollary}
Let $M$ be a 3-dimensional compact Riemannian manifold of nonnegative Ricci curvature with nonempty strictly convex boundary $\partial M$. Then any orientable properly embedded minimal surface $\Sigma$ in $M$ orthogonal to $\partial M$ is unknotted.
\end{corollary}

\begin{proof}
$M$ is diffeomorphic to the 3-ball $B^3$ by Theorem 2.11, \cite{FL}. ${\Sigma}$ divides ${M}$ into two handlebodies by Theorem \ref{convex-handlebodies}. Let $\check{M}$ be the doubling of $M$, and let $\check{\Sigma}$ be the doubling of $\Sigma$ in $M$. Then $\check{M}$ is diffeomorphic to $S^3$. By \cite{W} and \cite{HM} as in Corollary \ref{unknotted}, $\check{\Sigma}$ is unknotted in $\check{M}$ and $\Sigma$ is unknotted in $M$.
\end{proof}

We also have a version in the case of curvature with a negative lower bound.

\begin{theorem} \label{theorem:convex-negative}
Let $M^n$ be a compact Riemannian manifold with nonempty boundary. Suppose $M$ has Ricci curvature bounded from below $\mbox{Ric}_M \geq -(n-1)k$, $k >0$, and the boundary $\partial M$ is strictly convex with respect to the inward unit normal. Let $\Sigma$ be a hypersurface in $M$ that bounds a connected region $\Omega$ in $M$ and makes a constant contact angle $\theta\leq\pi/2$ with $\partial\Omega\cap\partial M$. Suppose that the mean curvature vector of $\Sigma$ points everywhere into $\Omega$, and $H \geq (n-1)\sqrt{k}$. Then $\Sigma$ is connected, and the map
\[
        i_*: \pi_1(\Sigma) \rightarrow \pi_1(\bar{\Omega})
\]
induced by the inclusion is surjective.
\end{theorem}

\begin{proof}
Suppose $\Sigma$ is not connected. Let $\Sigma_1$ and $\Sigma_2$ be two distinct connected components of $\Sigma$. Let $d_1$ and $d_2$ be the distance functions on $\Omega$ from $\Sigma_1$ and $\Sigma_2$ respectively. Since $\partial M$ is convex and $\Sigma_i$, $i=1,\, 2$,  makes a contact angle $\leq \pi/2$ with $\partial \Omega \cap \partial M$, for any point $x$ in $\Omega \setminus \Sigma_i$, $d_i(x)$ is realized by a geodesic in $\Omega$ from $x$ to an interior point $y$ on $\Sigma_i$. Then there exists a geodesic $\gamma$ in $\Omega$ from some interior point $p \in \Sigma_1$ to some interior point $q \in \Sigma_2$, that realizes the distance from $\Sigma_1$ to $\Sigma_2$, and meets $\Sigma_1$ and $\Sigma_2$ orthogonally. But as in the proof of Theorem \ref{theorem:negative} the Ricci curvature lower bound and assumption on the mean curvature of $\Sigma_1$ and $\Sigma_2$ imply that $\gamma$ is unstable, a contradiction. Therefore $\Sigma$ is connected.

Let $\tilde{\Omega}$ be the universal cover of $\Omega$ with projection map $\pi : \tilde{\Omega} \rightarrow \Omega$. By the same argument as above, $\partial \tilde{\Omega} \setminus \pi^{-1} (\partial M)$ must be connected, and we easily get the surjectivity of $i_*: \pi_1(\Sigma) \rightarrow \pi_1(\bar{\Omega})$.
\end{proof}

\begin{corollary}
Under the assumptions of Theorem \ref{theorem:convex-negative}, if $n=3$ then $\Omega$ is a handlebody.
\end{corollary}

\section{nonexistence}

As an application of the surjectivity of $i_*:\pi_1(\Sigma)\rightarrow\pi_1(M)$ Frankel showed that $\mathbb S^n$ cannot be minimally embedded in $\mathbb P^{n+1}$. In this section we further utilize the surjectivity of $i_*$ and prove nonexistence of some minimal embeddings in $T^{n+1}$.

Meeks \cite{Mk} proved that a compact surface of genus 2 cannot be minimally immersed in a flat 3-torus $T^3$. He used the fact that the Gauss map of a minimal surface $\Sigma\subset T^3$ into $\mathbb S^2$ has degree one. A  theorem of a similar nature can be proved in higher dimension by using the surjectivity of  $i_*$.

\begin{theorem}\label{generator}
Let $N$ be a compact orientable $n$-dimensional manifold with the number of generators of $\pi_1(N)=k$ and let $T^{n+1}$ be the $(n+1)$-dimensional flat torus. \\
 {\rm a)} If $k<n$, $N$ cannot be minimally embedded in $T^{n+1}$. \\
 {\rm b)} If $k=n$ and $N$ is minimally embedded in $T^{n+1}$, then $N$ is a flat $ T^n$.\\
 {\rm c)} If $k>n$ and $N$ is minimally embedded in $T^{n+1}$, then $N$ is separating and the number of generators of $\pi_1(D_j)$ must be bigger than $n$ for $j=1,2$ {\rm (}$D_1\cup D_2=T^{n+1}\setminus N${\rm )}.
\end{theorem}

\begin{proof}
Let $N^n$ be an embedded minimal submanifold in $T^{n+1}$ with $k\leq n$. Then the map $i_*:\pi_1(N)\rightarrow\pi_1(T^{n+1})$ is not surjective. Hence from Theorem \ref{heegaard} we conclude that $N$ is nonseparating and totally geodesic in $T^{n+1}$. Hence $N$ is a flat $ T^n$ and $k=n$. Therefore $N$ cannot be minimally embedded in $T^{n+1}$ in case $k<n$. If $k>n$, then $N$ must be separating and c) follows from the surjectivity of $i_*: \pi_1(D_j)\rightarrow \pi_1(T^{n+1})$ in Theorem \ref{heegaard} b).
\end{proof}

\begin{remark} In case $n=2$, Theorem \ref{generator} c) gives a new proof of the Meeks theorem mentioned above.
\end{remark}

Let $\Gamma_k\subset\mathbb R^{2}$ be the union of $k$ loops $\gamma_1,\ldots,\gamma_k$ in $\mathbb R^{2}$ with $\gamma_i\cap\gamma_j=\{p\}$ for every pair $1\leq i,j\leq k$ and let $\Gamma_k^{n+1}$ be the $\varepsilon$-tubular neighborhood of $\Gamma_k$ in $\mathbb R^{n+1}$. $\Gamma_k^{n+1}$ can be seen as a high-dimensional handlebody in $\mathbb R^{n+1}$. Note that $\partial\Gamma_k^{n+1}$ is diffeomorphic to $\#_k(\mathbb S^{n-1}\times\mathbb S^1)$, the connected sum of $k$ copies of $\mathbb S^{n-1}\times\mathbb S^1$, and that $\pi_1(\partial\Gamma_k^{n+1})$ has $k$ generators when $n\geq3$. Since $\partial\Gamma_n^{n+1}$ is not diffeomorphic to $T^n$, Theorem \ref{generator} implies that $\partial\Gamma_k^{n+1}$ cannot be minimally embedded in $T^{n+1}$ for any $k=1,\ldots,n$.

Schwarz's $P$-surface is a minimal surface of genus 3 in the cubic torus $T^3$. One can generalize this surface to higher dimension as follows. $T^{n+1}$ has a 1-dimensional skeleton $L_{n+1}$ which is homeomorphic to $\Gamma_{n+1}$. There also exists its dual $L_{n+1}'$ that is a parallel translation of $L_{n+1}$. One can foliate  $T^{n+1}\setminus(L_{n+1}\cup L_{n+1}')$ by 1-parameter family of $n$-dimensional hypersurfaces which are diffeomorphic to $\partial\Gamma_{n+1}^{n+1}$ and sweeping out from $L_{n+1}$ to $L_{n+1}'$. Applying the minimax argument, one could find a minimal hypersurface $\Sigma$ from this family of hypersurfaces \cite{CH}. $\Sigma$ should be diffeomorphic to $\partial\Gamma_{n+1}^{n+1}$ and $\pi_1(\Sigma)$ should have $n+1$ generators. Therefore the upper bound $n$ in Theorem \ref{generator} is sharp.

\end{document}